\documentclass[a4paper,12pt]{article} 
\usepackage{amsfonts,amssymb,amsthm,amsmath}
 \usepackage[dvips]{graphics}
\usepackage{color}

\DeclareSymbolFont{AMSb}{U}{msb}{m}{n}

\def\P{\mathbb P}
\def\N{\mathbb N}

\def\R{\mathbb R}
\def\C{\mathcal C}
\def\F{\mathcal F}
\def\E{\mathbb E}
\def\1{\,{\makebox[0pt][c]{\normalfont    1}
\makebox[2.5pt][c]{\raisebox{3.5pt}{\tiny {$\|$}}}
\makebox[-2.5pt][c]{\raisebox{1.7pt}{\tiny {$\|$}}}
\makebox[2.5pt][c]{} }}

\def\rR{r_R^{\phantom m}}

\def\for{\mbox{ for }}

\def\Pas{{\P\mbox{-a.s.}}}
\def\ol{\overline}
\def\dd{{\,\mathrm{d} }}
\def\ee{{\mathrm{e} }}
\def\bbox{{\hfill $\Box$}}

\newcommand{\norm}[1]{\left\| #1 \right\|}
\newcommand{\bignorm}[1]{\bigl\| #1 \bigr\|}

\newcommand{\mnorm}[1]{%
  \left\vert\kern-0.9pt\left\vert\kern-0.9pt\left\vert #1
    \right\vert\kern-0.9pt\right\vert\kern-0.9pt\right\vert
}
\newcommand{\bigmnorm}[1]{%
  \bigl\vert\kern-0.9pt\bigl\vert\kern-0.9pt\bigl\vert #1
    \bigr\vert\kern-0.9pt\bigr \vert\kern-0.9pt\bigr \vert
}
\newcommand{\smallfrac}[2]{ {\mbox{$\frac {#1}{#2}$}} }

\newtheorem{theorem}{Theorem}[section]
\newtheorem{lemma}[theorem]{Lemma}
\newtheorem{definition}[theorem]{Definition}

\newtheorem{remark}[theorem]{Remark}

\begin{document}

\title{{\Large  Existence and uniqueness of solutions of stochastic functional differential equations\\
}}
\author{ Max-K. von Renesse and Michael Scheutzow
\def\thefootnote{} \footnote{Technische Universit\"at Berlin, email: [mrenesse,ms]\@@math.tu-berlin.de}
}

\maketitle \abstract{Using a variant of the Euler-Murayama scheme for  stochastic functional differential equations
with bounded memory driven by Brownian motion we show that only weak  one-sided local Lipschitz (or 'monotonicity') conditions are
sufficient for local existence and uniqueness of strong solutions. In case of explosion the method  yields the maximal solution up to the explosion time. We also provide a weak growth condition which prevents explosions to occur. 
In an appendix we formulate and prove four lemmas which may be of independent interest: three of them
can be viewed as rather general stochastic versions of Gronwall's Lemma, the final one provides tail bounds for H\"older norms of stochastic integrals.} ~\\




\vspace{.5cm}

\section{Introduction}
There is by now a rather comprehensive mathematical literature on the mathematical theory and on applications of
stochastic functional (or delay) differential equations driven by Brownian motion. Existence and uniqueness of
global solutions have been established under global Lipschitz conditions
on the coefficients (e.g.\ \cite{Mohammed84}) or under local Lipschitz and linear growth conditions (e.g.\ \cite{Mao97, Xu08}).
On the other hand it is common knowledge for non-delay (stochastic) differential equations that only one sided Lipschitz conditions are
sufficient for local existence of solutions. This distinction becomes particularly relevant in infinite dimensions where the drift in (stochastic) evolution equations is unbounded and discontinuous in almost all interesting cases but nevertheless satisfies a one-sided Lipschitz i.e. 'monotonicity'/'dissipativity' condition, cf. e.g.\  \cite{PR07}. In this paper we show that monotonicity of the coefficients guarantees local existence of solutions to delay equations
 with bounded memory, thereby closing a systematic gap in the existing literature. \\
\smallskip

We choose the classical framework of the space of continuous
functions as a natural state space of the equation. Note that, due to the absence of an inner product on this space, 
the right formulation of monotonicity is not obvious in this case.  The  proposed condition \eqref{mono} below fits well to our
needs, since it recovers the classical monotonicity condition for the
non-delay case as a limit and yet is weak enough to cover  a rather big set of equations.  \smallskip

In our proof we define a specific Euler-Murayama scheme, which is generally a very
powerful tool in the Markovian case \cite{Aly87,K90,K99}. Other variants have been treated for the numerical simulation of stochastic delay equations under Lipschitz 
conditions in  e.g.\ \cite{cle06,ku00, klo07} and most recently \cite{bu08}. We point out that our method yields an approximation in the strong sense even in the case of  an explosion. In particular our proof below shows how the explosion time can be recovered  numerically, which seems to be a question typically neglected in the literature. 

\smallskip 
As for the proofs, note that the left hand side of condition  \eqref{mono} is quite weak w.r.t. the $C^0$-norm. As a consequence 
the standard two-step Burkholder-Davis-Gundy and Gronwall argument cannot be applied 
 to obtain the crucial contraction estimates. We overcome this difficulty by what we
call stochastic Gronwall lemmas and which are presented in the appendix. We
think that they may be of independent interest. These lemmas are
also crucial for the global existence assertion which holds under a
rather familiar growth (or 'coercivity',  \cite{PR07}) condition \eqref{coerc}, which is  again weak in  the $C^0$-topology.


\section{Set Up and Main Results}
 
For  $r>0$, let $\C$ denote the space of continuous $\R^d$-valued functions
on $[-r,0]$ endowed with the sup-norm $\|.\|$. For a function or a
process $X$ defined on $[t-r,t]$ we write $X_t(s):=X(t+s)$, $s \in
[-r,0]$.  Consider the stochastic functional differential equation
\begin{equation}
\label{sfde}
\left\{
\begin{aligned}
\dd X(t) & = f(X_t)\,\dd t + g(X_t) \dd W(t),\\
X_0& = \varphi,\\
 \end{aligned}
\right.
\end{equation}
where $W$ is an $\R^m$-valued Brownian motion  defined on a complete probability space $(\Omega,\F,\P)$ endowed with the augmented Brownian
filtration $\F^W_t = \sigma\left(W(u), {0\leq u \leq t} \right)\vee \mathcal N\subset \F$, where $\mathcal N$ denotes the null-sets in
$\F$, $\varphi$ is an $(\F^W_t)$-independent  $\C$-valued random variable and   $f:\C \to \R^d$, $g: \C \to \R^{d \times m}$ are continuous maps. \\

\noindent 
We will suppose throughout this work the following monotonicity assumption on  $f$ and $g$.
\begin{equation}\label{mono} \tag{M}
 \begin{minipage}{12cm}
For each compact subset $C \subset \C$, there exists a number $K_C$ and some $r_C \in ]0,r]$ such that for all $x,y \in C$ with $x(s)=y(s)\, \forall \, s \in [-r,-r_C]$
\[2\,\langle f(x)-f(y),x(0)-y(0)\rangle  + \mnorm{ g(x)-g(y)}^2 \le K_C\,\norm{x-y} ^2,\]
 \end{minipage}
\end{equation}
where    $\langle .,.\rangle$ denotes the standard inner product on $\R^d$ and  $\mnorm{M}^2 = \mathrm{tr} (M M^*)$ for $M \in \R^{d\times m}$. \\

\noindent As an example in $d=1$ take  $f(x) = \varphi(\sum_{i=1}^N w_i \, x(t_i))$, where $t_i \in [-r,0], w_i \geq 0$,  $i=1, \dots, N$  and $\varphi \in C(\R)$ is a non-increasing continuous (not necessarily Lipschitz) function, e.g. $\varphi(s) = -{\rm sign} (s) \sqrt{ |s|}$ and $g$ locally Lipschitz on $\C$. Another example is  $f= f_1 +f_2 +f_3$ with    $f_1$ locally Lipschitz on $\C$, $f_2(x) = \int_{-r}^{-r_0} \psi(x(s)) k(s) ds$ for some $0<r_0<r$, $k, \psi \in C(\R)$  and $f_3(x)= \varphi(x(0))$ with $\varphi \in C(\R)$ non-increasing as above. \\

Our first result is a local existence and uniqueness statement for solutions to (\ref{sfde}) for which we recall some basic notions. Given any filtration $(\F_t)$ on $\Omega$, an $(\F_t)$-stopping time   $\sigma: \Omega \to \overline  \R_{\geq 0}$ is called {\em predictable} if there exists a sequence of (`announcing') stopping times $\sigma_n$ such that $\sigma_n < \sigma$ and $\sigma_n \nearrow \sigma $  $\P$-almost surely.
A tuple $X=(X,\sigma)$ of
a predictable stopping time $\sigma$ and a map $X:\Omega \times ( [-r,0]\cup [0,\sigma[ )\to \R^d$
is called a {\em local $(\F_t)$-semimartingale up to time $\sigma$ starting  from
$\varphi \in \C$}, if $X_0 =\varphi$ holds  $\P$-almost surely and for any (announcing)
stopping time $\sigma_n < \sigma$, the process $(X^{\sigma_n}(t))_{t\geq 0}$ with $X^{\sigma_n}(t)=X({t\wedge \sigma_n})$ is an $\R^d$-valued $(\F_t)$-adapted semimartingale.
\begin{definition}[Local Solution] \label{solutiondef}\textit{Let $\F_t= \F_t^W \vee \sigma(\varphi)$.  A local $(\F_t)$-semimartingale $(X,\sigma)$ up to a
predictable stopping time $\sigma$  is called a {\em local strong solution} to equation (\ref{sfde}) if $X_0 =\varphi$ and for any
stopping time $\sigma_n < \sigma$ and any $t \geq 0$
\[ X({t\wedge\sigma_n}) = X(0) + \int_{0} ^{t\wedge \sigma_n} f(X_u) \dd u + \int_{0}^{t\wedge \sigma_n} g(X_u)\dd W(u)  \quad \Pas
\]
The pair $(X,\sigma)$ is called {\em maximal strong solution} if  in addition $(X_t)$ eventually leaves  any compact set $K \subset \C$ for $t \to \sigma$,  $\P\mbox{-almost surely on } \{\sigma < \infty \}$; i.e.\
$$
\P \left(\{\exists \mbox{ a compact set } K \subset \C \mbox{ and } t_i \nearrow \sigma \mbox{ s.t. } X_{t_i} \in K\} \cap  \{\sigma < \infty \}\right )=0.
$$}
\end{definition}

\begin{theorem} \label{local} \textit{Equation \eqref{sfde} admits a unique maximal strong solution $(X,\sigma)$ provided \eqref{mono} holds.
}
\end{theorem}

\begin{theorem} \label{global} \textit{In addition to the assumptions of Theorem \ref{local} let  $f$ and $g$ be  bounded on bounded subsets of $\C$ and let the pair $(f,g)$ be weakly coercive in the sense that there exists a non-decreasing function $\rho: [0,\infty[ \to ]0,\infty[$ such that
$\int_0^{\infty} 1/\rho(u) \dd u = \infty$ and for all $x\in \C$
\begin{equation} \label{coerc}\tag{C}
2\,\langle f(x),x(0)\rangle  + \mnorm{g(x)}^2 \le \rho (\norm{x}^2).
\end{equation}
Then $X$ is globally defined, i.e.\ $\sigma  = \infty$ $\P$-almost surely.}
\end{theorem}

\section{Proof of Theorem \ref{local}}

The proof of Theorem \ref{local} is based on an iteration of Lemma \ref{verylocal} below, which requires some auxiliary notation.
For $\Phi \subset \C$ and $R>0$ let
\[ C_{\Phi,R} = \{ \eta \in  \C | \,\exists\,   \varphi \in \Phi, r_0 \in [0,r]:  \eta(u) = \varphi(u+r_0), u\in ]-r,-r_0], \norm{\eta-\varphi(0)}_{1/4;[-r_0,0]}\leq R\} \subset \C, \]
where
\[\norm{\eta}_{\alpha;[a,b]}= \sup_{a\leq u<v\leq b} \bigl(|\eta(v)-\eta(u)|/(v-u)^{\alpha}\bigr)+\sup_{a\leq u \leq b}  |\eta(u) | \]
denotes the H\"older-$\alpha$-norm on $C([a,b],\R^d)$, $\alpha \in (0,1)$.
Note that $C_{\Phi,R}$ is compact in $\C$ provided $\Phi$ is. \\

Below we  drop the subscript $\Phi$ whenever this causes no confusion.

\begin{lemma}\label{verylocal}\textit{In addition to the conditions of Theorem \ref{local} assume there is a compact subset $\Phi \subset \C$ such that $\varphi \in \Phi$ $\P$-almost  surely.  For $R>0$, let $\rR=r_C$ be the constant appearing in \eqref{mono} for choosing $C=C_{\Phi,R}$. Then there exists a stopping time $0<\sigma_R\leq \rR$ and a unique (up to indistinguishability) $(\F_t)$-adapted process $X(t)$, $t \in [0,\sigma_{R}]$ such that $X_t \in C_{R} $ for all $t \in [0,\sigma_{R}]$ which solves  \eqref{sfde} up to time $\sigma_{R}$. Moreover,
\begin{equation}\label{mindestens}
\norm{X(.)-\varphi(0)}_{1/4;[0,\sigma_R]} \geq  \frac R2  \quad \P\mbox{-a.s. on } \{ \sigma_R< \rR\}.
\end{equation}
}
 \end{lemma}
\begin{proof}
The proof is inspired by the arguments for finite dimensional monotone SDEs in
\cite{K99}, cf.\ e.g. \cite{PR07}.
For $n \in \N$,  we define an Euler-like  approximation to \eqref{sfde} with step size $\frac 1 n$  by
\begin{equation}
\label{qesfde}
\left\{
\begin{aligned}
\dd X^{n}(t) & = f(\ol X^{n}_t)  \dd t + g(\ol X^{n}_{t}) \dd W(t),\\
X^{n}_0& = \varphi,\\
 \end{aligned}
\right.
\end{equation}
where we define $\ol X^n_s(.) \in \C$, $s\geq 0$  by
\[
 \ol X^n _s(u) = X^n \bigl((s+u)\wedge\smallfrac{\lfloor n s\rfloor }{n}\bigr), \quad u \in [-r,0].\]
Equation (\ref{qesfde}) admits a global in time solution  via  the recursion $X^n_0 = \varphi$ and
\begin{align*}
X^{n}(t)&= X^{n}\bigl ( { \mbox{$\smallfrac {\lfloor n t \rfloor }{n} $}}\bigr) + \int_{{\lfloor n t \rfloor }/{n}}^t f \bigl( \ol X^{n}_s\bigr) \dd s
 + \int_{{\lfloor n t \rfloor }/{n}}^ t g \bigl ( \ol X^{n}_s \bigr ) \dd W(s).
\end{align*}
The process $t\mapsto X^{n}(t)$ is adapted and continuous, hence
$$
t \mapsto p_t^{n}(.):= \ol X^{n}_{t}(.) - X^{n}_t(.),\; t \ge 0
$$
defines an adapted $\C$-valued process (which is c\`adl\`ag). With this, \eqref{qesfde} is equivalent to $X_0^n=\varphi$ and
$$
X^{n}(t)=\varphi(0) + \int_0^t f(X^{n}_s + p_s^{n} )\dd s +    \int_0^t g(X^{n}_s + p_s^{n} ) \dd W(s).
$$
Without loss of generality, we may assume that the set $\Phi$ has the property that $0 \in \Phi$ and
$\eta \in \Phi,\,s \in [-r,0)$ implies that the function $u \mapsto \eta(u \wedge s),\,u \in [-r,0]$ also belongs to $\Phi$.
Then, $X^n_t \in C_R$ implies $\ol X^n _t \in C_R$, hence
 $p^{n}_t \in   \widetilde C_{R}= \{\eta_1-\eta_2\,|\, \eta_i \in C_R\}$ provided
  \[
t \leq   \tau^{n}_R:=\inf\{t > 0 | X^n_t \notin C_R\}.
\]
Since $\widetilde C_R \subset \C$ is again compact,
\begin{equation}
\widetilde \rho(R)=\sup_{x\in  \widetilde C_{R}} \norm{x} < \infty \label{sizebound}
 \end{equation}
and the continuity of  $f$ and $g$ ensures that
\begin{equation}  C_1(R):=\sup_{ x \in  \widetilde C_R } \{|f(x)|+ \mnorm{g(x)}\}<\infty. \label{coeffbound}
 \end{equation}
Fix $n,m \in \N$ and  let $0\leq \tau$ be a finite stopping time. Then, by It\^o's formula,
\begin{align}
|&X^{n}(\tau)  - X^{m}(\tau)|^2
=  2\int_0^\tau \langle X^n(u)-X^m(u),\left(g(X^{n}_u + p_u^{n} )-g(X^{m}_u + p_u^{m} )\right) \dd W(u) \rangle \nonumber \\
& +  \int_0^\tau \Bigl(2 \langle  f(X^{n}_u + p_u^{n} )-f(X^{m}_u + p_u^{m} ), X^{n}(u) - X^{m}(u)\rangle
+ \bigmnorm{g(X^{n}_u + p_u^{n} )-g(X^{m}_u + p_u^{m} )}^2 \Bigr)\dd u.
\nonumber  \end{align}
In order to use condition (\ref{mono}), note that by construction for $s>0$ and $s+u \leq 0$
\[ \ol X^m_s(u) = \ol X^n_s(u)=\varphi(s+u).   \]
Hence, together with  (\ref{sizebound})  and (\ref{coeffbound}),  the second term on the r.h.s. can be estimated from above by
\begin{align*}
&\int_0^\tau \left( 2\langle  f(X^{n}_u + p_u^{n} )-f(X^{m}_u + p_u^{m} ),  p^{m}_u(0) -p^{n}_u(0)\rangle +   K_R \bignorm{X^{n}_u+p^{n}_u -(X^{m}_u +p^{m}_u) } ^2 \right) \dd u \\
&\leq  \int_0^\tau \left( 4C_1(R) \bigl(|p_u^{n}(0)|+| p_u^{m}(0)|\bigr)   +  4 K_R\bigl( \bignorm{p^{n}_u} ^2+ \bignorm{p^{m}_u} ^2 \bigr)+2 K_R \bignorm{X^{n}_u-X^{m}_u}^2 \right) \dd u  \\
&\leq  \int_0^\tau  \bigl[ 4C_1(R)+4K_R\widetilde \rho(R)\bigr] \bigl(\bignorm{p_u^{n}}+\bignorm{ p^{m}_u}\bigr)   +  2 K_R \sup_{v \in [0,u]} |{X^{n}(v)-X^{m}(v)}|^2 \dd u
\end{align*}
provided  $\tau \leq \tau^{m}_R\wedge \tau^{n}_R\wedge \rR=:\kappa$. Hence we
may apply Lemma \ref{thelem} to $Z(s):= |X^{n}(s\wedge \kappa)-X^{m}(s\wedge \kappa)|^2$ with  $M(s) :=
2\int_0^{s \wedge \kappa} \langle X^n(u)-X^m(u),\left(g(X^{n}_u + p_u^{n}
  )-g(X^{m}_u + p_u^{m} )\right) \dd W(u) \rangle$, $H(s) = \int_0^{s
  \wedge \kappa}   \bigl[ 4C_1(R)+4K(R)\widetilde \rho(R)\bigr]
  \bigl(\bignorm{p_u^{n}}+\bignorm{ p^{m}_u}\bigr)  \dd u$ and $T=\rR$.
Once we have shown that some moment of $H^*(T):=\sup_{0\le s \le T} H(s)$ converges to 0 as $n,m \to \infty$,
Lemma \ref{thelem} implies that for all $\varepsilon>0$,
\begin{equation}
\lim_{m,n\to \infty}\P\big\{\sup_{s \in [0, \tau^{m}_R\wedge\tau^{n}_R\wedge \rR]} |X^{m}(s)-X^{n}(s)| \ge \varepsilon\big\} =0.
\label{firstlimstatement}
\end{equation}
Since $H^*(T)$ is bounded uniformly in $\omega,n,m$, it suffices to show that $H^*(T)$ converges
to zero in probability as $m,n \to \infty$ which can be verified as follows:
\[
 p^n _s(u) =
\left\{
\begin{array}{ll}
 0 & \for u \geq -r, u +s \leq \frac{\lfloor ns \rfloor }{n}\\
-\int_{\lfloor ns  \rfloor /n }^{s+u} f(\ol X^n _t)\dd t -  \int_{\lfloor ns \rfloor  /n }^{s+u} g(\ol X^n _t)\dd W(t) & \for u +s\geq  \frac{\lfloor ns \rfloor }{n}, u \leq 0
\end{array}\right.
\]
implies
\begin{equation*}
 \norm{p^n_s} \leq  \sup_{\lfloor n s\rfloor / n\leq t\leq s }
 \Bigl|\int_{\lfloor n s\rfloor / n}^t f(\ol X^{n}_{u}) \dd u \Bigr| +
\sup_{\lfloor n s\rfloor / n\leq t\leq s}  \Bigl|\int_{\lfloor n s\rfloor / n}^t  g(\ol X^{n}_{u})\dd W(u) \Bigr|, \label{triplest}
\end{equation*}
and hence -- since $f$ and $g$ are bounded on $C_R$ --
$$
\E \1_{\{\tau_R^n \ge s\}}\norm {p_s^n}  \to 0 \mbox{ as } n \to \infty, \mbox{ uniformly in } [0,r_R].
$$
Therefore, $\E H^*(T)$ converges to 0 and \eqref{firstlimstatement} follows.
By definition of $\ol X^m$ this also yields
\begin{equation}
\lim_{m,n\to \infty}\P\big\{\sup_{s \in [0, \tau^{m}_R\wedge\tau^{n}_R\wedge \rR]} \norm{\ol X^{m}_s -\ol X^{n}_s}  \ge \varepsilon \big\} =0.
\label{interpolimstatement}
\end{equation}
Since $f, g$ are uniformly continuous on the compact set $C_R$ 
\begin{equation}
\lim_{m,n\to \infty}\P\big\{\sup_{s \in [0,\tau^{m}_R\wedge\tau^{n}_R\wedge\rR]}
  \big\{|f(\ol X^m_s)-f(\ol X^n_s)|\vee \bigmnorm{g(\ol X^m_s)-g(\ol X^n_s)}\big\}  \ge \varepsilon \big\} =0.
\label{seclimstatement}
\end{equation}
To further improve this statement, we apply Lemma \ref{dereichlem}   to
  \begin{align}
 X^{n}(s\wedge \tau^m_R \wedge \tau^n_R  \wedge \rR)-  X^{m}(s\wedge \tau^m_R \wedge \tau^n_R  \wedge \rR)& =
\int_0 ^{s\wedge \tau^m_R \wedge \tau^n_R  \wedge \rR} (F^{n} -F^{m})(u) \dd Z(u),  \nonumber
\end{align}
where for simplicity we write $Z(u) = (u, W(u)) \in \R^{m+1}$ and  $F^n(u) = \bigl(f(\ol X^{n}_{u}), g(\ol X^{n}_{u})\bigr)$. Together with   (\ref{seclimstatement})  this allows to
 conclude  that  for all $\varepsilon > 0$
 \begin{equation}
\lim_{m,n\to \infty}\P\big\{ \bignorm{X^{m}(.)-X^{n}(.)}_{1/4;[0,  \tau^{m}_R \wedge\tau^{n}_R\wedge \rR]}  \ge \varepsilon\big\} =0.
\label{thirdlimstatement}
\end{equation}
Let us select a subsequence, which will again be denoted by $X^{n}$ such that
\begin{equation}
\P\bigl\{\bignorm{X^{k}(.)-X^{l}(.)}_{1/4;[0, \tau^{k}_R \wedge \tau^l_R\wedge \rR]}   \ge 2^{-(l\wedge k)} \bigr\}\leq 2^{-(l\wedge k)},
\label{subseqchoice}
\end{equation}
and define
\[ \tau_R = \liminf_{n\to \infty} \tau^{n}_R. \]
Due to \eqref{subseqchoice}, there is an $(\F_t)$-adapted process $X$ defined in $[0,\tau_R[ \cap [0,\rR]$ to which $X^n$ converges $\P$-almost surely locally in $C^{1/4}([0,\tau_R[ \cap [0,\rR];\R^d)$. From \eqref{qesfde}, \eqref{firstlimstatement} and \eqref{seclimstatement} and the continuity
of $f$ and $g$ we infer that  $X$ must be a solution to equation $(\ref{sfde})$ on $[0,\tau_R[ \cap [0,\rR[$.\\

We remark that   $\tau_R >0$ almost surely, which can be seen as follows. For any  $\varepsilon >0$, using \eqref{subseqchoice} we choose $n_0$ such that the set $$
 A= \big\{\omega|\sup_{k \geq n_0} \bignorm{X^{n_0}(.)-X^k(.)}_{1/4;[0,\tau_R^k \wedge \tau_R^{n_0} \wedge \rR]}
< \frac R2\big\} $$ satisfies $\P(A)\geq 1-\varepsilon$. From $\ol X^{n_0}_s(.)=\varphi((s+\cdot)\wedge 0) \in \Phi$
for $s\in [0,\frac 1 {n_0}]$, using Lemma \ref{dereichlem} for the SDE \eqref{qesfde} solved by $X^{n_0}$,
it follows that   $\eta^{n_0}_{R/2} := \inf\big\{t\geq 0 \,|\norm{X^{n_0}(.)-\varphi(0)}_{1/4;[0,t]} \ge \frac R2\big\}\wedge \rR$ is strictly positive. By construction of $A$ it holds on $A$ that $\tau^n_R\wedge \rR \geq \eta^{n_0}_{R/2}\wedge \rR$ for all $n\geq n_0$, hence in particular $\tau_R >0$.\\

Next, we show that almost surely one of the two following events occur:
\begin{equation}\label{eitheror}
\{\tau_R\geq \rR\} \qquad \mbox { or } \qquad \big\{\tau_R < \rR\big\} \cap \big\{\sup_{t<\tau_R} \norm{X(.)-\varphi(0)}_{1/4;[0,t]} \ge \frac{3R} 4\big\}.
\end{equation}
In  case $\{\tau_R\geq \rR\}$, using  \eqref{sfde} for $X(.)$ on $[0, \rR[$ and the uniform boundedness of the coefficients on
$C_R$ we may extend $X(.)$ on the closed interval $[0,\rR]$ by setting
\[  X(\rR) := X(0) + \int_0^{\rR} f(X_s) ds + \int_0^{\rR} g(X_s) dW(s).\]
Together with    \eqref{eitheror}   for
$$
\sigma_R:=\inf\big\{t\in [0,\tau_R[ \cap [0,\rR] \,|\norm{X(.)-\varphi(0)}_{1/4;[0,t]} \ge \frac R2\big\} \wedge \rR$$
this gives a well defined process $t\mapsto X(t)$ for $t\in  [0,\sigma_R]$ which solves \eqref{sfde} in up to time $\sigma_R$ in the sense of Definition \ref{solutiondef}. Moreover, \eqref{mindestens} holds by construction.\\

To prove  \eqref{eitheror} we show that the set
$$
B:=\{\tau_R < \rR\} \cap  \big\{\sup_{t<\tau_R} \norm{X(.)-\varphi(0)}_{1/4;[0,t]} < \frac {3 R} 4\big\}.
$$ has vanishing $\P$-measure. Assume the contrary, i.e. $\P(B)=p>0$. Then by \eqref{subseqchoice} and the definition of $\tau_R$  we find some $n_0 \in \N$ such that $ \P (A) > \frac p 2 $,  where
$$A:=\{\omega|\sup_{k \geq n_0} \bignorm{X^{n_0}(.)-X^k(.)}_{1/4;[0,\tau_R^k \wedge \tau_R^{n_0} \wedge \rR]}  < \frac R{16};\, \inf_{n\geq n_0} \tau^{n}_R < \rR ; \,  \sup_{t<\tau_R} \norm{X(.)-\varphi(0)}_{1/4;[0,t]} < \frac {3 R} 4 \big\}.$$
We show that in fact $\P(A)=0$. To this aim note that w.l.o.g.\ we may assume
that $X^n$ converges to  $X$ locally in $C^{1/4}([0,\tau_R[)$ and  $[0, \rR\wedge \tau^m _R]\ni t \mapsto \norm{X^{m}(.)}_{1/4;[0,t]}$ is continuous for all $m \in \N$, for \textit{all} $\omega \in A$, where the latter is again a consequence of Lemma \ref{dereichlem}. Now for   $\omega \in A$ choose $m=m(\omega) \geq n_0$ such that $\tau^m_R < \rR$. Let $\eta^{m}_R := \inf\big\{t\geq 0 \,|\norm{X^{m}(.)-\varphi(0)}_{1/4;[0,t]} \ge R \big\} \leq \tau_R^m$,
then by continuity  $\eta^{m}_{7R/8}< \eta^{m}_{15R/16} \leq \tau^n_R$ for all $n\geq n_0$, hence $\eta^{m}_{7R/8}  < \tau_R$.   Again by continuity,  $\sup_{t< \eta^m _{7R/8}} \norm{X^n(.)-\varphi(0)}_{1/4;[0,t]} \geq  \frac {3 R} 4$ for all $n \geq n_0$ satisfying $\tau^n_R > \eta^m_{7R/8}$. In view of the convergence of $X^n$ to $X$ in $C^{1/4}[0,\eta^{m}_{7R/8}]$ for $n\to \infty$ this yields a contradiction to $\sup_{t<\tau_R} \norm{X(.)-\varphi(0)}_{1/4;[0,t]} < \frac {3 R} 4$. Hence $A= \emptyset$ almost surely which proves   \eqref{eitheror}.\\

To show uniqueness of a local solution, assume $X$
and $\tilde X$ are two solutions defined up to a stopping time
$\tilde \sigma  \le \sigma_R$. Applying It\^o's formula to the square of the norm of the difference of the solutions and using condition
(M), Lemma \ref{thethirdlem} (with $C=0$) shows that the solutions agree on $[0,\tilde \sigma]$ almost surely.
This completes the proof of Lemma \ref{verylocal}.
\end{proof}

\noindent
\textit{Proof of Theorem \ref{local}.} First we remark that it is sufficient to prove both the existence and uniqueness assertion of the theorem under the stronger assumption that $\P(\varphi \in \Phi)=1$ for any fixed  compact  subset $\Phi\subset \C$. In fact, since any probability measure on the Polish space $\C$ is tight, in both cases the general statement  follows  by approximation in $\P$-measure by initial conditions $\varphi_n = 1_{\Phi_n}(\varphi)\cdot \varphi$, where e.g. the compact subsets $\Phi_n \subset \C$ are chosen such that $\P(\varphi\not \in \Phi_n) \leq \frac 1 n $.\smallskip

The proof of the existence statement is based on iterative use of Lemma \ref{verylocal}. Recall for $R>0$, $\rR$ denotes the constant  $r_C$ in condition \eqref{mono} when $C=C_{\Phi,R}$. We may assume w.l.o.g. that the function $R \mapsto \rR$ is non-increasing and
we may select a sequence $R^{(k)}\nearrow \infty$, $k \in \N$, such that $ {\sum}_{k} r_{\!R^{(k)}} =\infty$.

 \smallskip

Lemma \ref{verylocal} with $\Phi=:\Phi^{(1)}$ and $R:= R^{(1)}$ for initial condition $\varphi=:\varphi^{(1)}\in \Phi^{(1)}$ guarantees the existence of a process $t\mapsto X(t)=:X^{(1)}(t)$, $t \in [0,\sigma^{(1)}]$ with an $\F_\cdot$-stopping time $\sigma^{(1)}:=\sigma_{R^{(1)}}\leq r_{R^{(1)}}$ which is a local solution to \eqref{sfde} on $[0,\sigma^{(1)}[$. \smallskip

 Next we may apply Lemma \ref{verylocal} to the same equation \eqref{sfde}, now  in the situation when
$R$ and $W$ are chosen to
be   $R^{(2)}$ and $W^{(2)}_t=W(\sigma^{(1)} +t)-W(\sigma^{(1)})$  on $(\Omega, \F, \P)$ respectively, with
$\F^{(2)}_t = \F^{W^{(2)}}_t \vee \mathcal N\subset \F$, $(t\geq 0)$, and $\F^{(2)}_\cdot$-independent initial
condition $\varphi^{(2)} := X^{(1)}_{\sigma^{(1)}} \in C_{\Phi,R_1}=:\Phi^{(2)}$. This yields an
$\F^{(2)}_\cdot $-stopping time $\tilde \sigma^{(2)}\leq r_{R^{(2)}}$ and a process
$ t \mapsto \tilde X^{(2)}$, $[0,\tilde \sigma^{(2)}]$ solving  \eqref{sfde}  on  $t \in [0,\tilde \sigma^{(2)}[$.
(Note that here we have used the simple  fact that  $  C_{C_{\Phi,R_1},R_2} = C_{\Phi,R_2}$ for $R_2 \geq R_1$.)
Hence,  by continuation
\[
 t \mapsto X^{(2)} (t) =
\left\{
\begin{array}{ll}
 X^{(1)}(t) & \mbox{if } t \in [-r,\sigma^{(1)}]\\
\tilde X^{(2)}(t-\sigma^{(1)}) & \mbox{if } t \in ]\sigma^{(1)}, \sigma^{(1)}+\tilde \sigma^{(2)}]
\end{array}
\right.
\]
we obtain an $\F_\cdot$-adapted process which is a local solution to equation \eqref{sfde} up to the $\F_\cdot$-stopping time $\sigma^{(2)} = \sigma^{(1)} + \tilde\sigma^{(2)}$ in the sense of Definition \ref{solutiondef}. \\

 For general $n$ this construction is repeated inductively, furnishing a local solution $(X,\sigma)$ to equation \eqref{sfde} in the sense of Definition \ref{solutiondef} where
\[\sigma = \lim_{n\to \infty } \sigma^{(n)}.\]
To  prove  that   $(X,\sigma)$ is maximal   using the continuity of   $f$ and $g$ it suffices to prove that the set
\begin{equation}\label{Sigma}
\Sigma =    \bigl \{\sup_{t\in [0,\sigma[}(|f(X_t)|\vee \mnorm{g(X_t)})<\infty\bigr\}\cap \bigl\{\sigma < \infty \bigr\}
\end{equation}
has zero $\P$-measure. Now from  the second statement in Lemma \ref{verylocal}, from
the construction of $X$ and from the property  $\sum_{k} r_{\!R^{(k)}} =\infty$ it follows that
\[ \sup_{s\in [0,\sigma[}\,\norm{X(.)-X(\sigma^{(k-1)})}_{1/4;[0,s]}  \geq \frac{R^{(k)}} 2  \quad \Pas \]
for infinitely many  $k \in \N$     on $ \{\sigma < \infty\}$, i.e.
\[
 \P(\Sigma\bigr)=
\P\Bigl( \sigma<\infty;
\sup_{s\in [0,\sigma[}(|f(X_s)|\vee \mnorm{g(X_s)})<\infty
;\sup_{s\in [0,\sigma[}\,\norm{X(.)}_{1/4;[0,s]} =\infty  \Bigr).
\]
Since $X$ solves \eqref{sfde}, due to e.g.\  Lemma \ref{dereichlem},
the r.h.s.\ is zero. \\

As for the uniqueness statement let  $(Y,\tau)$ be another maximal solution with an associated sequence of announcing stopping times $\tau^{(n)}$. The construction of $X$ above yields a sequence of announcing stopping times $\sigma^{(n)}$ for $\sigma$  and compact sets $C_n \subset \C$ such that $X_{t\wedge \sigma^{(n)}} \in C_n$. Hence, by the same argument as in the proof of Lemma \ref{verylocal} one obtains that $X_{\sigma^{(n)}\wedge\tau^{(n)}\wedge\cdot}$ and  $Y_{\sigma^{(n)}\wedge\tau^{(n)}\wedge\cdot}$ are indistinguishable. Moreover, the maximality of the pair $(Y,\tau)$ implies  that $\sigma^{(n)} < \tau$ for all $n \in \N$, i.e.\ $\sigma \leq \tau$ almost surely. Conversely, the maximality of $\sigma$ implies $\sigma > \tau^n$, i.e $\sigma \geq \tau$, which completes the proof. \bbox

%

\section{Proof of Theorem \ref{global}}

\textit{Proof of Theorem \ref{global}.}
Let $(X,\sigma)$ be the maximal strong solution of equation \eqref{sfde}. We want to show that $\sigma=\infty$ almost surely.
Since $f$ and $g$ are bounded on bounded subsets of $\C$, it follows from \eqref{Sigma} that
$\limsup_{t \nearrow \sigma} |X(t)|=\infty$ almost surely on the set $\{\sigma < \infty\}$.
For a stopping time $0 \le \tau  < \sigma$, It\^o's formula implies that
\begin{align*}
X^2(\tau)-X^2(0)&=\int_0^{\tau} 2\langle f(X_u),X(u)\rangle + \mnorm{g(X_u)}^2 \dd u +
2 \int_0^{\tau} \langle X(u),g(X_u) \dd W(u) \rangle\\
&\le \int_0^{\tau} \rho(\norm{X_u}^2 ) \dd u + M(\tau),
\end{align*}
where $M$ is a continuous local martingale. Applying Lemma \ref{theseclem} to $Z(t):=X^2(t)$ finishes the proof.
\bbox

\section{Appendix}

We start by proving three lemmas which could be called {\em stochastic Gronwall lemmas}. We use them in the proof of Theorems  \ref{local} and \ref{global}. Then we prove a result about the tails of H\"older norms of stochastic integrals which
we owe to Steffen Dereich (TU Berlin). We believe that all these results are of independent interest. In all lemmas, we assume that
a filtered probability space $(\Omega,\F,(\F_t)_{t \ge 0},\P)$ is given and that it satisfies the usual conditions.
Throughout, we will use the notation $Z^*(T)=\sup_{0 \le t \le T} Z(t)$ for a real-valued process $Z$.

\begin{lemma}\label{theseclem}
Let $\sigma>0$ be a stopping time and let $Z$ be an adapted non-negative stochastic process with continuous paths defined
on $[0,\sigma[$ which satisfies the inequality
\begin{equation*}
Z(t) \le \int_0^t \rho(Z^*(u)) \dd u + M(t) + C,
\end{equation*}
and $\lim_{t \uparrow \sigma} Z^*(t)= \infty$ on $\{\sigma < \infty\}$ almost surely. Here, $C \ge 0$ and $M$ is a continuous
local martingale defined
on $[0,\sigma[$, $M(0)=0$ and $\rho:[0,\infty[ \to ]0,\infty[$ is non-decreasing,  and $\int_0^{\infty} 1/\rho (u) \dd u = \infty$.
Then $\sigma = \infty$ almost surely.
\end{lemma}

\begin{proof}
Let $Y$ be the unique (maximal) solution of the equation
$$
Y(t) =  \int_0^t \rho(Y^* (u)) \dd u + M(t)+C.
$$
Clearly, $Y(t) \ge Z(t)$ for all $t$ for which $Y$ is defined and therefore it suffices to prove the claim for $Y$ instead of $Z$.
For $a>C$, define $\tau_a:=\inf\{t \ge 0| Y(t) \ge a\}$.
For $C<a<b$ and $\delta >0$ we get
$$
\P\{\tau_b-\tau_a \le \delta|\F_{\tau_a}\} \le \P\{b-a \le \delta \rho(b) + \sup_{t \in [\tau_a,\tau_b \wedge (\tau_a + \delta)]}
M(t)-M(\tau_a)| \F_{\tau_a}\}
$$
on the set $\{\tau_a < \infty\}$. Note that on $\{\tau_a < \infty\}$ we have
\begin{equation}\label{lowerest}
M(t)-M(\tau_a) \ge Y(t)-Y(\tau_a) - (t-\tau_a)\rho(b) \ge -a-\delta \rho(b)
\end{equation}
for $\tau_a \le t \le \tau_b \wedge (\tau_a + \delta)$ since $Y$ is non-negative. For
$$
\tau:=\inf\{t \ge \tau_a| M(t)-M(\tau_a) \geq b-a-\delta \rho(b)\}\wedge \tau_b \wedge (\tau _a + \delta)
$$
we therefore get
$$
0=\E(M(\tau)-M(\tau_a) | \F_{\tau_a}) \ge (b-a-\delta \rho(b))p - (a + \delta \rho(b))(1-p),
$$
where $p:=\P\{M(\tau)-M(\tau_a)\geq b-a-\delta \rho(b)|\F_{\tau_a}\}$. Hence
\begin{equation}\label{ab}
\P\{\tau_b-\tau_a \le \delta|\F_{\tau_a}\} \le p \le \frac{a+\delta \rho(b)}{b} \mbox{ on } \{\tau_a<\infty\}.
\end{equation}

Fix $a> C$. Then
$$
\sigma= \tau_a + \sum_{k=1}^{\infty} \big( \tau_{2^k a} - \tau_{2^{k-1} a} \big).
$$
We show that the sum diverges almost surely. To ease notation, we write $\tau_k$ instead of $\tau_{2^k a}$.
For $\delta_k>0$, $k \in \N$, \eqref{ab} implies that
$$
\P\{\tau_k-\tau_{k-1} \ge \delta_k|\F_{\tau_{k-1}}\} \ge \frac{1}{2} - \delta_k \frac{\rho(2^k a)}{2^k a}
$$
on the set $\{\tau_{k-1} < \infty\}$. Now
\begin{equation}\label{sigma}
\sigma \ge \sum_{k=1}^{\infty}  \tau_k - \tau_{k-1}  \ge \sum_{k=1}^{\infty}  \delta_k 1_{\{\tau_k - \tau_{k-1} \ge \delta_k\}}.
\end{equation}
We choose
$$
\delta_k:=\frac{1}{4} \frac{2^k a}{\rho(2^k a)}\;k \in \N.
$$
Since $\rho$ is non-decreasing we have
$$
\sum_{k=1}^{\infty}  \delta _k \ge \frac{1}{4} \int_{2a}^{\infty} \frac{1}{\rho(u)} \dd u = \infty
$$
and
$$
\P\{\tau_k-\tau_{k-1} \ge \delta_k|\F_{\tau_{k-1}}\} \ge \frac{1}{4} \;\mbox{ on } \{\tau_{k-1} < \infty\}.
$$
It follows (e.g.\ from Kolmogorov's three series theorem) that the right hand side of \eqref{sigma} diverges on the set
$\{\tau_k<\infty \mbox{ for all } k \in \N\}$. On the complement of this set, $\sigma$ is also infinite, i.e. the proof of
the lemma is complete. \end{proof}

While the previous lemma was concerned with non-blow up of $Z$, the following lemma shows that $Z$ remains small it case
the initial condition is small. In principle we could formulate the following lemma also using a function $\rho$ as in the
previous one but we prefer not to in order to obtain a reasonably explicit formula for moments of $Z^*(T)$.

\begin{lemma}\label{thethirdlem}
Let $Z$ be an adapted non-negative stochastic process with continuous paths defined
on $[0,\infty)$ which satisfies the inequality
\begin{equation*}
Z(t) \le K \int_0^t Z^*(u) \dd u + M(t) + C,
\end{equation*}
where $C \ge 0$, $K>0$ and $M$ is a continuous local martingale with $M(0)=0$.
Then for each $0<p<1$, there exist universal finite constants $c_1(p), c_2(p)$ (not depending on $K,C,T$ and $M$)
such that
$$
\E (Z^*(T))^p \le C^p c_2(p) \exp\{c_1(p) KT\} \mbox{ for every } T \ge 0.
$$
\end{lemma}

\begin{proof}
Let $Y$ be the unique solution of the equation
$$
Y(t) = K \int_0^t Y^* (u) \dd u + M(t) + C.
$$
Clearly, $Y(t) \ge Z(t)$ for all $t \ge 0$ and therefore it suffices to prove the claim for $Y$ instead of $Z$.
Let $\tau_a:=\inf\{t \ge 0: Y(t) \ge a\}$. Like in the proof of Lemma \ref{theseclem}, we obtain for $\beta \in (0,1)$ and
$b>a\ge C$
\begin{equation}\label{ungl}
\P\big\{\tau_b -\tau_a \le \frac{\beta}{K} | \F_{\tau_a}\big\} \le \frac{a+\beta b}{b} \mbox{ on } \{\tau_a<\infty\}.
\end{equation}
For $T>0$, $m \in \N$, $\gamma>(1-\beta)^{-1}$ we get
$$
\P\{Y^*(T) \ge \gamma^m C\} = \P\{\tau_{\gamma^m C} \le T\}=\P\{\sum_{i=1}^m \tau_{\gamma^iC}-\tau_{\gamma^{i-1}C} \le T\}.
$$
By \eqref{ungl}, the last sum is stochastically larger than $\beta/K$ times a binomial variable $V$ with parameters $m$ and
$\alpha:=1-\frac{1}{\gamma}-\beta$. Therefore, for $\lambda>0$ and $N:=\lceil \frac{KT}{\beta} \rceil$ we get
$$
\P\{Y^*(T) \ge \gamma^m C\} \le \P\{V \le N\} = \P\{\ee^{-\lambda V} \ge \ee^{-\lambda N}\}.
$$
Applying Markov's inequality, representing $V$ as a sum of $m$ independent Bernoulli($\alpha$) variables and optimizing over
$\lambda >0$ as usual, we obtain for $m\ge \lceil \frac{N}{\alpha} \rceil = :m_0$
$$
\P\{Y^*(T) \ge \gamma^m C\} \le \exp\{(m-N)\log  \frac{m}{m-N} + (m-N)\log (1-\alpha) + N \log \alpha + N \log \frac{m}{N}\}.
$$
Assume that $p\log \gamma + \log(1-\alpha)<0$ (which requires $p<1$ since $1-\alpha =\frac{1}{\gamma}+ \beta>\frac{1}{\gamma}$) and
fix $q>0$ such that $ p\log \gamma + \log(1-\alpha) + q^{-1} <0 $. Then
\begin{eqnarray*}
\E Y^*(T)^p &=& \int_0^{\infty} \P\{ Y^*(T) \ge s^{1/p} \} \dd s \\
&\le& \gamma^{m_0 p} C^p + \sum_{m=m_0}^{\infty}C^p \gamma^{pm}(\gamma-1)
\exp\Big\{(m-N)\log  \frac{m}{m-N} + (m-N)\log (1-\alpha) \\
&& \hspace{2cm} + N \log \alpha + N \log \frac{m}{N}\Big\}\\
&\le& \gamma^{m_0 p} C^p  + C^p (\gamma-1)\exp\{N \log \frac{\alpha q}{1-\alpha}\} \sum_{m=m_0}^{\infty}
\exp\{m(p\log \gamma + \log(1-\alpha)+q^{-1})\}\\
&=& C^p \Big(\gamma^{m_0 p} + (\gamma-1)\exp\{N \log \frac{\alpha q}{1-\alpha}\} \frac{\exp\{m_0(p\log \gamma + \log(1-\alpha)+q^{-1})\} }
{1-\exp\{p\log \gamma + \log(1-\alpha)+q^{-1}\}}\Big),
\end{eqnarray*}
where we used the inequalities $\log (1+x) \le x$ (for $x= \frac {N}{m-N}$) and $\log x \le \log q + q^{-1} (x-q)$ (for $x= \frac m N $) in the last ``$\le$''.
Observing that $m_0 \le (\frac{kT}{\beta}+1)\frac{1}{\alpha}+1$ and $N \le \frac{KT}{\beta}+1$, the claim follows.
\end{proof}

\begin{remark} {\normalfont
It is clear that the previous lemma does not hold for $p > 1$: just consider a scalar geometric Brownian motion starting with
$C$. Its $p^{\mathrm{th}}$ moment for $p > 1$ at time 1 (say) is unbounded with respect to the volatility $\sigma$. We don't know
whether the lemma holds true for $p=1$ but we conjecture that it doesn't.}
\end{remark}

\begin{lemma}\label{thelem}
Let $Z$ be an adapted non-negative stochastic process with continuous paths defined
on $[0,\infty[$ which satisfies the inequality
\begin{equation*}
Z(t) \le K \int_0^t Z^*(u) \dd u + M(t) + H(t),
\end{equation*}
where $K>0$, $M$ is a continuous local martingale with $M(0)=0$, and
$H$ is an adapted process with continuous paths satisfying $H(0)=0$.
Then, for each $0<p<1$ and $\alpha>\frac{1+p}{1-p}$,  
there exist  constants $c_3, c_4$ depending on
$p,\alpha$ only such that
$$
\E (Z^*(T))^p \le c_3 \exp\{c_4 KT\} (\E H^*(T)^{\alpha})^{p/\alpha} \mbox{ for every } T \ge 0.
$$
\end{lemma}

\begin{proof}
Fix $T>0$ and for $i \in \N$ let $X_i$ be the unique solution of
$$
X_i(t)= K\int_0^t X_i^*(u) \dd u + M(t) + i. $$
Hence, $Z \le X_i$ on $[0,T] \times \Omega_i$ where
$$
\Omega_i:=\{\omega: \sup_{0 \le t \le T} H(t) \le i\}.
$$
 Let $s \in ] \frac{1}{1-p}, \frac{\alpha}{1 + p}[$ and let $r>1$ be defined by $r^{-1}+s^{-1}=1$. Then $pr<1$ and Lemma
\ref{thethirdlem} and H\"older's inequality imply
\begin{eqnarray*}
\E (Z^*(T))^p &\le&    \sum_{i=1}^{\infty} \E ((X_i^*(T))^p \1_{\Omega_i \backslash \Omega_{i-1}}) \le \sum_{i=1}^{\infty} (\E (X_i^*(T))^{pr})^{1/r} \P\{\Omega_i \backslash \Omega_{i-1}\}^{1/s}\\
&\le& \sum_{i=1}^{\infty} i^p c_2(pr)^{1/r} \exp\{ KTc_1(pr)/r\} \P\{H^*(T) \ge i-1\}^{1/s}\\
&\le& \exp\{ KTc_1(pr)/r\} c_2(pr)^{1/r} \Big( (\E H^*(T)^{\alpha})^{1/s}  \sum_{i=2}^{\infty} i^p(i-1)^{-\alpha/s} + 1\Big),
\end{eqnarray*}
where we used Markov's inequality in the last step. 

For each $\xi>0$, the inequality in the assumption of the lemma remains true if $H, \,M,$ and $Z$ are multiplied by $\xi$.
Therefore, the inequality
$$
\E (Z^*(T))^p \le \exp\{ KTc_1(pr)/r\} c_2(pr)^{1/r} \Big(\xi^{\frac{\alpha}{s}-p} (\E H^*(T)^{\alpha})^{1/s}  \sum_{i=2}^{\infty} i^p(i-1)^{-\alpha/s} + \xi^{-p} \Big)
$$
follows. Optimizing the right hand side over $\xi>0$ yields the assertion of the lemma.
\end{proof}

\begin{lemma}[\textbf{S.\ Dereich}]\label{dereichlem} \textit{For $m,d \in \N$, $\alpha \in ]0,\frac 1 2 [$
and $t_0>0$ there exist some universal
strictly positive constants $ c_i = c_i(d,m,\alpha,t_0), i=1,2,3$  such that for $Z(t)=(t,W(t)) \in \R^{m+1}$ with an
$\R^m$-valued Brownian motion  $W$
\[  \mathbb P\Bigl( \bignorm{ \int^{(.)}_{\sigma} F \dd Z}_{\alpha;[\sigma,\tau]}\geq u  \Bigr) \leq
c_1 e^{- c_2 u^2/v^2   T} \quad \mbox{ for } \frac u{v(T+T^{1-\alpha})} \geq c_3,\;T \geq t_0\]
for any pair  $\sigma\leq  \tau $ of finite $(\F_t)$-stopping times with $\tau -\sigma \leq T$ and any
$(\F_t)$-predictable $\R\times \R^{d\times m}$-valued process $(F(t))$ satisfying
$\sup_{s\in [\sigma,\tau]}\mnorm{F(s)} \leq v$ $\P$-almost surely.}
\end{lemma}

\begin{proof} It suffices to treat the
case when $\sigma=0$ and $m=d=1$,  where we have to deal with real-valued semimartingales  of the form
\[t \mapsto  \int_0^t F(s) \dd s=:A(t) \qquad \mbox{ or } \qquad t \mapsto \int_0^t F(s) \dd W(s) =: M(t) \]
with  integrands satisfying $\sup_{s\in [0,T]} |F(s)|\leq v$  almost surely. The first case is easy: the map
$t \mapsto A(t)$ is Lipschitz with constant (at most) $v$ and therefore
$\norm{A(.)}_{\alpha;[0,\tau]} \leq v   \,( T+ T^{1-\alpha})$ almost surely, so the claim follows in this case.
Let us consider $M$. The Gaussian isoperimetric inequality, cf. e.g.   \cite[Section 4.3]{Boga98},
implies the existence of  some universal positive constants $k_i = k_i(\alpha), i=1,2$  such that
\[  \mathbb P\Bigl( \bignorm{ W(.)}_{\alpha;[0,1]}\geq u  \Bigr) \leq k_1 e^{- k_2 u^2}
\quad \mbox{ for } u \geq 0.\]
We choose an independent Brownian motion $W'$ and let
$F'(s) = \sqrt { v^2 - F^2(s)}$. Then both  processes 
\[ t \mapsto B^{(j)}(t) =\int _0 ^ t F(s)\dd W(s) -(-1)^{j}\int_0^t F'(s)\dd W'(s), \quad j = 1 ,2, \]
have the same distribution as $ t\mapsto v W(t)$. From  $B^ {(1)}(t) + B^{(2)}(t) = 2 \int_0^t F(s)\dd W(s)$
and the triangle inequality in $C^\alpha$ one gets
\begin{eqnarray*}
\mathbb P\Bigl( \bignorm{ \int^{(.)}_0F(s)\dd W(s)}_{\alpha;[0,\tau]}\geq u  \Bigr)
&\leq& 2 \P \Bigl( \bignorm{v W(.)}_{\alpha;[0,T]}\geq  u   \Bigr) \leq 2 \P \Bigl( \bignorm{ W(.)}_{\alpha;[0,1]}
\geq \frac{u}{v \sqrt{T} (T^{-\alpha} \vee 1)}  \Bigr)  \\   &\leq& 2 k_1 \exp\big\{- k_2 \frac{u^2}{v^2   T (T^{-\alpha} \vee 1)^2}\big\},
\end{eqnarray*}
which yields the claim of the lemma.
\end{proof}
 

\noindent {\bf Remark}: Alternatively, the previous lemma can be proved using the fact that each continuous local martingale starting
at 0 can be represented as a time-changed Brownian motion.
%
%

\end{document}